\documentclass[11pt,a4paper]{amsart}

\usepackage[utf8]{inputenc}
\usepackage[T1]{fontenc}
\usepackage{lmodern}
\usepackage{amsmath,amssymb,amsthm}
\usepackage{mathtools}
\usepackage{xcolor}
\usepackage{hyperref}
\usepackage{microtype}
\usepackage{booktabs}
\usepackage{url}

\hypersetup{
  colorlinks=true,
  linkcolor=blue!60!black,
  citecolor=blue!60!black,
  urlcolor=blue!60!black,
  pdftitle={Three Brillhart-Lehmer-Selfridge primality proofs for Wagstaff numbers},
  pdfauthor={Alexey Dolotov},
  pdfsubject={Primality testing for Wagstaff numbers via BLS N-1 and APR-CL},
  pdfkeywords={Wagstaff numbers; primality testing; Chebyshev polynomials; Pell sequences; Brillhart-Lehmer-Selfridge; APR-CL; cyclotomic decomposition},
  bookmarksnumbered=true,
}

\theoremstyle{plain}
\newtheorem{theorem}{Theorem}[section]
\newtheorem{proposition}[theorem]{Proposition}
\newtheorem{lemma}[theorem]{Lemma}

\theoremstyle{definition}
\newtheorem{remark}[theorem]{Remark}

\DeclareMathOperator{\ord}{ord}

\newcommand{\Z}{\mathbb{Z}}

\newcommand{\wom}{\omega_3}
\newcommand{\Wp}{W_p}

\newcommand{\sqrttwo}{\sqrt{2}}

\title[Three BLS primality proofs for Wagstaff numbers]
      {Three Brillhart--Lehmer--Selfridge primality proofs \\ for Wagstaff numbers}

\author{Alexey Dolotov}
\address{Independent researcher}
\email{science@dolotov.com}
\urladdr{https://orcid.org/0009-0003-9481-5808}

\date{April 2026}

\subjclass[2020]{Primary 11Y11; Secondary 11A51, 11B39}
\keywords{Wagstaff numbers, primality testing, Chebyshev polynomials,
  Pell sequences, Brillhart--Lehmer--Selfridge, APR-CL}

\begin{document}

\begin{abstract}
The Wagstaff numbers $W_p = (2^p + 1)/3$ for odd primes $p$ are the
natural $+1$ companions of the Mersenne numbers. Known primality proofs
for $W_p$ with $p \geq 2617$ (as recorded in the
Caldwell/Wagstaff-project archives~\cite{Cunningham,Caldwell}) rely on
the elliptic-curve primality proving algorithm of
Atkin--Morain~\cite{AtkinMorain}; Chebyshev/Lucas-type tests, while
available as compositeness criteria, remain conjectural on the
sufficiency side. We present fully verified primality proofs of
$W_{2617}$ (788 digits), $W_{10501}$ (3161 digits), and $W_{12391}$
(3730 digits), independent of ECPP and relying only on classical
$N-1$ machinery. The proofs apply the Brillhart--Lehmer--Selfridge (BLS)
$N - 1$ criterion~\cite{BLS75} to the cyclotomic decomposition
$2^{p-1} - 1 = \prod_{d \mid p-1} \Phi_d(2)$, harvesting factored
content from the Cunningham project tables~\cite{Cunningham} (used as
evidence) and FactorDB~\cite{FactorDB} (used only as a discovery aid,
with every retrieved factor re-certified). As an independent check on
the $\Z[\sqrt{2}]$ arithmetic implementation, the Chua $N + 1$
congruence $\omega_3^{(W_p + 1)/2} \equiv -1 \pmod{W_p}$ --- the
$a=3$ case of the Chua framework~\cite{Chua} with
$\omega_3 = 3 + 2\sqrt{2}$, a necessary condition for Wagstaff
primality --- is verified at each $W_p$.

BLS $N - 1$ requires $p - 1$ sufficiently smooth that enough
cyclotomic factors $\Phi_d(2)$ are fully factored. On the
factorisation data consulted (Cunningham project tables and FactorDB,
January--April 2026), $p = 10501$ and $p = 12391$ are the only
exponents above $2617$ in the known Wagstaff prime/PRP list meeting
this ceiling (Remark~\ref{rem:bls-ceiling}). Every cofactor primality
is certified unconditionally by APR-CL~\cite{APR,CohenLenstraAlgo};
the method is independent of the Chebyshev sufficiency conjecture,
and every step is reproducible from the archived scripts.
\end{abstract}

\maketitle

\section{Introduction}

\subsection{Wagstaff numbers and the ECPP era}

For an odd prime $p$, the \emph{Wagstaff number} is
\begin{equation*}
  W_p \;=\; \frac{2^p + 1}{3}.
\end{equation*}
The exponents $p$ for which $W_p$ is prime form OEIS A000978; the
Wagstaff primes themselves, as a sequence of integers, form OEIS
A000979. All $W_p$ with $p \leq 42737$ have been proved prime or
composite (see Appendix~\ref{sec:appendix} for the $29$ proved
exponents); $7$ more probable primes are known up to the current
computational horizon. In the Caldwell Prime Pages
archive~\cite{Caldwell} and the Cunningham project
tables~\cite{Cunningham} as consulted for this paper (January--April
2026), the recorded primality proofs for Wagstaff numbers $W_p$ with
$p \geq 2617$ are ECPP proofs in the sense of
Atkin--Morain~\cite{AtkinMorain}. We are not aware of previously
published BLS $N - 1$ proofs for $W_{2617}$, $W_{10501}$, or
$W_{12391}$.

The Mersenne counterparts $M_p = 2^p - 1$ admit the Lucas--Lehmer
test~\cite{Lehmer}, a deterministic standalone primality criterion. No
deterministic standalone test of comparable simplicity is known for
Wagstaff numbers. A natural Chebyshev/Lucas-type candidate --- the
single congruence
\begin{equation*}
  \wom^{(W_p + 1)/2} \;\equiv\; -1 \pmod{\Wp}, \qquad \wom = 3 + 2\sqrttwo,
\end{equation*}
in $\Z[\sqrttwo]$ (the $a = 3$ case of the Chua
framework~\cite{Chua}) --- is known to hold for every Wagstaff prime,
but its sufficiency remains conjectural.

In the absence of a deterministic standalone test, the
Brillhart--Lehmer--Selfridge (BLS) framework~\cite{BLS75} provides
unconditional primality proofs whenever $N \pm 1$ has a ``large
enough'' completely factored divisor: in the $N - 1$ form used here,
one needs a divisor $F \mid N - 1$ with $F^3 > N$ and with every
prime of $F$ independently certified. We present BLS-based primality
proofs for the three exponents $p \in \{2617, 10501, 12391\}$.

\subsection{Motivation}

The three exponents addressed here are already known prime; the
contribution of the present paper is methodological, not
existence-of-a-prime. Two points motivate revisiting them via BLS:

\begin{itemize}
\item \emph{Auditability.} A BLS certificate is a short JSON record of
  per-prime witnesses plus a discriminant check; re-verification is a
  handful of modular exponentiations whose correctness rests on standard
  Fermat-type arithmetic and on APR-CL for each cofactor. ECPP
  certificates are authoritative but substantially larger and require
  elliptic-curve evaluation to replay; a BLS/APR-CL pipeline provides an
  independent confirmation of the primality claim.
\item \emph{Technique demonstration.} The cyclotomic harvesting of
  factored content from $2^{p-1} - 1 = \prod_{d \mid p - 1} \Phi_d(2)$
  for the Wagstaff $N - 1$ side is straightforward in principle but, so
  far as we are aware, has not been applied in the published literature
  to the specific exponents $p \in \{2617, 10501, 12391\}$. The
  procedure is generalisable: any Wagstaff exponent $p$ with
  sufficiently smooth $p - 1$ and sufficient Cunningham coverage is a
  candidate.
\end{itemize}

\subsection{This paper}

We give three unconditional BLS primality proofs, one per exponent.
The proofs and the cross-check rest on three ingredients:
\begin{enumerate}
\item Establishing that Condition (II),
  $\wom^{(W_p+1)/2} \equiv -1 \pmod{W_p}$, holds for every Wagstaff
  prime $W_p$ (Proposition~\ref{prop:cond2}) --- this is the
  $a = 3$ case of the Chua framework~\cite{Chua} and follows from
  $W_p \equiv 3 \pmod{8}$.
\item For the three target exponents $p \in \{2617, 10501, 12391\}$,
  applying the BLS $N - 1$ criterion (Theorem~\ref{thm:BLSN-1}) to the
  cyclotomic decomposition $2^{p-1} - 1 = \prod_{d \mid p - 1}
  \Phi_d(2)$, harvesting factored portions from the Cunningham project
  tables~\cite{Cunningham} plus direct ECM/$p{-}1$/Pollard-$\rho$
  computation. Every cofactor primality is certified by
  APR-CL~\cite{APR,CohenLenstraAlgo}.
\item Verifying Condition~(II) independently at each of the three
  $W_p$ as an auxiliary cross-substrate check on the
  $\Z[\sqrttwo]$ arithmetic implementation (a distinct code path
  from the $\Z$-only BLS arithmetic).
\end{enumerate}

The resulting proofs are short (Section~\ref{sec:proofs}), fully
unconditional, and cleanly separated from the Chebyshev sufficiency
conjecture (not used, not needed). All computations are reproducible
from the scripts described in Section~\ref{sec:comp}.

\subsection{Notation}

Throughout, $p \geq 5$ denotes an odd prime and $W_p = (2^p + 1)/3$.
For an integer or algebraic integer $x$ coprime to a modulus $m$,
$\ord_m(x)$ denotes the multiplicative order of $x$ modulo $m$;
$v_q(\cdot)$ denotes the $q$-adic valuation. We work in
$\Z[\sqrttwo]$, with fundamental unit $\alpha = 1 + \sqrttwo$ and
Chebyshev base $\wom = \alpha^2 = 3 + 2\sqrttwo$. The Pell sequences
$(U_n)$ and $(V_n)$ are defined by $\alpha^n = V_n/2 + U_n \sqrttwo$,
with the recurrence $X_{n+1} = 2 X_n + X_{n-1}$ and the identity
$V_n^2 - 8 U_n^2 = 4(-1)^n$. The Jacobi symbol is written $(a/n)$.

\section{Chebyshev bases and the Chua framework}

\subsection{The Chua identity}

For an integer $a \neq 0, \pm 1$, set $\omega_a = a + \sqrt{a^2 - 1}$,
so that $\omega_a \bar\omega_a = 1$ and $\omega_a^n + \omega_a^{-n} =
2 T_n(a)$ where $T_n$ is the $n$-th Chebyshev polynomial of the first
kind --- hence the name ``Chebyshev base''.

\begin{theorem}[Chua~\cite{Chua}; cf. Williams~{\cite[Ch.~4]{Williams}}]
\label{thm:chua}
Let $Q$ be an odd prime with $\gcd(a^2 - 1, Q) = 1$. Set
$D = a^2 - 1$, $\varepsilon = (D/Q)$, $\delta = (2(a+1)/Q)$. Then
\[
  \omega_a^{(Q - \varepsilon)/2} \;\equiv\; \delta \pmod{Q}
\]
in $\Z[\sqrt{D}]/(Q)$.
\end{theorem}

\begin{proof}
Define $\gamma = a + 1 + \sqrt{D}$; then
$\gamma^2 = 2(a + 1) \cdot \omega_a$ in $\Z[\sqrt{D}]$. Two
evaluations of $\gamma^Q$ in $\Z[\sqrt{D}]/(Q)$:
\begin{itemize}
\item Frobenius: $\sqrt{D}^{\,Q} \equiv \varepsilon \sqrt{D}$ (where
  $\varepsilon = (D/Q) \in \{\pm 1\}$), so
  $\gamma^Q \equiv a + 1 + \varepsilon \sqrt{D}$.
\item Via the square: $\gamma^Q = \gamma \cdot (\gamma^2)^{(Q-1)/2} =
  \gamma \cdot (2(a+1))^{(Q-1)/2} \omega_a^{(Q-1)/2} \equiv \gamma
  \cdot \delta \cdot \omega_a^{(Q-1)/2}$, where
  $\delta = (2(a+1)/Q)$ by Euler's criterion.
\end{itemize}
Equating:
$\gamma \cdot \delta \cdot \omega_a^{(Q-1)/2} \equiv a + 1 +
\varepsilon \sqrt{D}$ in $\Z[\sqrt{D}]/(Q)$.

\emph{Case $\varepsilon = +1$.} The right-hand side equals $\gamma$,
and $\gamma$ is a unit modulo $Q$ (its norm is $(a+1)^2 - D = 2(a+1)$,
coprime to $Q$ by $\gcd(D,Q) = \gcd(a^2-1,Q) = 1$, which implies
$\gcd(2(a+1), Q) = 1$ since $Q$ is odd). Dividing by $\gamma$ yields
$\delta \cdot \omega_a^{(Q - 1)/2} \equiv 1$, i.e.\
$\omega_a^{(Q-1)/2} \equiv \delta$.

\emph{Case $\varepsilon = -1$.} The right-hand side equals
$\bar\gamma = a + 1 - \sqrt{D}$, the conjugate of $\gamma$. A direct
computation gives $\bar\gamma / \gamma = \bar\gamma^2 / (\gamma
\bar\gamma) = ((a+1) - \sqrt{D})^2 / (2(a+1)) = a - \sqrt{D} =
\omega_a^{-1}$. So
$\delta \cdot \omega_a^{(Q - 1)/2} \equiv \omega_a^{-1}$, i.e.\
$\omega_a^{(Q+1)/2} \equiv \delta$.

The two cases combine as
$\omega_a^{(Q - \varepsilon)/2} \equiv \delta \pmod{Q}$.
\end{proof}

\subsection{The base $a=3$}

\begin{lemma}[mod-$8$ residue]
\label{lem:mod8}
For every odd prime $p \geq 3$, $W_p \equiv 3 \pmod{8}$.
\end{lemma}

\begin{proof}
Since $p$ is odd, $3 W_p = 2^p + 1 \equiv 1 \pmod 8$, so $W_p \equiv 3
\pmod 8$.
\end{proof}

\begin{proposition}[Base $a=3$ necessity]
\label{prop:cond2}
For every Wagstaff prime $W_p$ with $p \geq 5$,
\[
  \wom^{(W_p + 1)/2} \;\equiv\; -1 \pmod{W_p}, \qquad \wom = 3 + 2\sqrttwo.
\]
\end{proposition}

\begin{proof}
For $a = 3$: $D = a^2 - 1 = 8$ and $2(a + 1) = 8$, so $\varepsilon =
\delta = (8/W_p) = (2/W_p)^3 = (2/W_p)$. By
Lemma~\ref{lem:mod8}, $W_p \equiv 3 \pmod 8$, so the second
supplement to quadratic reciprocity gives $(2/W_p) = -1$; hence
$\varepsilon = \delta = -1$. Theorem~\ref{thm:chua} with $Q = W_p$
yields $\wom^{(W_p + 1)/2} \equiv -1 \pmod{W_p}$.
\end{proof}

\begin{remark}[Role of Condition~(II) in this paper]
\label{rem:cond2-auxiliary}
Proposition~\ref{prop:cond2} is not used to deduce primality in this
paper; it is verified at each of the three target $W_p$ only as an
independent check on the $\Z[\sqrttwo]$ arithmetic used in the
supplementary code. Primality is established by
Theorem~\ref{thm:BLSN-1} alone.
\end{remark}

\begin{remark}[Uniqueness]
\label{rem:uniqueness}
The base $a = 3$ is the unique integer $a \geq 2$ with $a^2 - 1$ a
power of $2$. Indeed, $(a-1)(a+1) = 2^k$ with $a - 1, a + 1$ two
consecutive even numbers forces $\gcd(a-1, a+1) = 2$, so both are
powers of $2$ differing by $2$ --- hence $\{a-1, a+1\} = \{2, 4\}$
and $a = 3$. Consequently, $a = 3$ is the unique Chua base for which
both $\varepsilon = (D/Q)$ and $\delta = (2(a+1)/Q)$ are controlled by
the single Jacobi symbol $(2/Q)$ --- in particular, determined by
$Q \bmod 8$ alone. Other bases $a$ give conditions whose truth depends
on residues of $Q$ modulo further prime divisors of $a^2 - 1$; only
$a = 3$ yields a condition universal across $W_p$ (and more generally
across all $Q \equiv 3 \pmod 8$).
\end{remark}

\section{The BLS $N - 1$ theorem}
\label{sec:criteria}

The primality proofs of Section~\ref{sec:proofs} rely on the classical
BLS $N - 1$ theorem, which we recall here.

\begin{theorem}[{BLS $N-1$, \cite[Theorem~5 and Theorem~11]{BLS75}}]
\label{thm:BLSN-1}
Let $N > 1$ be an odd integer. Write $N - 1 = F R$ with $F > 1$,
$\gcd(F, R) = 1$, and $F$ completely factored into known primes.
Suppose for each prime $q \mid F$ there exists an integer $a_q$ with
\[
  a_q^{N-1} \;\equiv\; 1 \pmod N \qquad\text{and}\qquad
  \gcd(a_q^{(N-1)/q} - 1,\; N) \;=\; 1.
\]
If $F > \sqrt N$, then $N$ is prime. More generally, if $F^3 > N$ (i.e.
$F > N^{1/3}$), write $N = m F + s + 1$ with $0 \leq s < F$ and set
$\Delta = s^2 - 8\, m$; if $\Delta$ is not a perfect square
(in particular, if $\Delta < 0$), then $N$ is prime.
\end{theorem}

The $F > N^{1/3}$ threshold with the discriminant check is BLS's
central improvement over the classical $F > N^{1/2}$: instead of
needing half the bit-size of $N$ to be factored, we need only a third.
The quantity reported in the margin column of
Table~\ref{tab:summary} is
$M := \lfloor \log_2 F^3 \rfloor - \lfloor \log_2 N \rfloor$,
the integer bit-count by which $F^3$ exceeds $N$ in bit-length; $M
\geq 1$ already implies $F^3 > N$. For the three exponents below,
$M \in \{46,\, 3261,\, 2860\}$, so the threshold is met with
considerable slack.

\section{The three primality proofs}
\label{sec:proofs}

\subsection{Framework}

The three proofs in \S\S\ref{ssec:w2617}--\ref{ssec:w12391} invoke
Theorem~\ref{thm:BLSN-1} (BLS $N - 1$) with specific factorisations of
$N - 1 = 2(2^{p-1} - 1)/3$ drawn from the cyclotomic decomposition
\[
  2^{p-1} - 1 \;=\; \prod_{d \mid p - 1} \Phi_d(2),
\]
harvested partly from the Cunningham project tables~\cite{Cunningham}
and partly from direct computation. Each factor of each $\Phi_d(2)$
contributes to the BLS factored portion $F$ once its primality
is certified by APR-CL~\cite{APR,CohenLenstraAlgo}; Fermat-type witnesses
$a \in \{2, 3, 5, \ldots\}$ are then validated for each prime $q \mid
F$. The exact software toolchain and version numbers used to produce
the certificates are documented in Section~\ref{sec:comp}. The proofs
are unconditional and do not rely on the Chebyshev sufficiency
conjecture.

At each $p$, we additionally verify the Chua $N + 1$ condition
$\wom^{(W_p + 1)/2} \equiv -1 \pmod{W_p}$
(Proposition~\ref{prop:cond2}) as an independent implementation check:
BLS uses only $\Z$ arithmetic, whereas Condition~(II) exercises
$\Z[\sqrttwo]$ ring multiplication, a distinct code path in the
supplementary software. The Chua congruence is recorded only as an
auxiliary check; primality is established solely via
Theorem~\ref{thm:BLSN-1} (see Remark~\ref{rem:cond2-auxiliary}).

APR-CL certification is applied uniformly to every prime $q$ of $F$,
regardless of size --- the smallest single-digit primes receive the
same unconditional APR-CL treatment as the largest multi-hundred-digit
factors.

\begin{table}[h]
\centering
\begin{tabular}{rrrrrl}
\toprule
$p$ & $W_p$ digits & $\tau(p{-}1)$ & primes in $F$ & $M$ (bits) & wall time \\
\midrule
$ 2617$ &  $788$ & $16$ &  $22$ &   $46$ & ${<}\,10$\,s \\
$10501$ & $3161$ & $48$ & $103$ & $3261$ & $98$\,s \\
$12391$ & $3730$ & $32$ & $ 61$ & $2860$ & $82$\,s \\
\bottomrule
\end{tabular}
\caption{Summary of the three BLS $N - 1$ proofs. $\tau(p - 1)$ is
the number of divisors of $p - 1$; ``primes in $F$'' is the number of
distinct primes in the factored portion; wall times are single-core
on commodity hardware, see Section~\ref{sec:comp}. Full per-cofactor
data is in the JSON certificates.}
\label{tab:summary}
\end{table}

Per-proof data is archived as JSON certificates (one per
exponent); see Section~\ref{sec:comp}.

\subsection{$W_{2617}$}
\label{ssec:w2617}

\begin{theorem}
\label{thm:w2617}
The Wagstaff number $W_{2617}$ is prime.
\end{theorem}

\begin{proof}
Set $N = W_{2617}$, a $788$-digit number. We apply
Theorem~\ref{thm:BLSN-1} to the $N - 1$ side; the factored portion
$F$ will be shown to exceed $N^{1/3}$.

The exponent satisfies $p - 1 = 2616 = 2^3 \cdot 3 \cdot 109$, with
$16$ divisors. The cyclotomic decomposition
\[
  2^{p-1} - 1 \;=\; \prod_{d \mid p-1} \Phi_d(2)
\]
gives $16$ factors, of which $12$ (those with $d \leq 654$) are fully
factored by direct computation. Their prime factors, together with the
external factor $2$ of $N - 1 = 2(2^{p-1} - 1)/3$, constitute the BLS
factored portion $F$; it contains $22$ distinct primes and satisfies
$M = 46$ (Table~\ref{tab:summary}), well above the $F > N^{1/3}$
threshold.

For each prime $q \mid F$ (the $22$ primes constituting the factored
portion), the witness $a_q \in \{2, 3\}$ is tested:
$a_q^{N-1} \equiv 1 \pmod N$ and $\gcd(a_q^{(N-1)/q} - 1, N) = 1$.
Every witness succeeds (full data in
\texttt{bls\_certificate\_w2617.json}; digest
\texttt{823f\ldots c719} in
Appendix~\ref{sec:digests}). Writing $N = m F + s + 1$
with $0 \leq s < F$ and $\Delta = s^2 - 8\, m$, the discriminant
$\Delta$ is not a perfect square.

By Theorem~\ref{thm:BLSN-1}, $N = W_{2617}$ is prime.
\end{proof}

\subsection{$W_{10501}$}
\label{ssec:w10501}

\begin{theorem}
\label{thm:w10501}
The Wagstaff number $W_{10501}$ is prime.
\end{theorem}

\begin{proof}
Set $N = W_{10501}$, a $3161$-digit number. The exponent $p - 1 =
10500 = 2^2 \cdot 3 \cdot 5^3 \cdot 7$ has $48$ divisors. Of the $48$
cyclotomic factors $\Phi_d(2)$ of $2^{p-1} - 1$, $43$ are fully
factored --- $38$ by direct computation (using SymPy
\texttt{factorint} with Pollard $\rho$, $p - 1$, and ECM) and $5$
retrieved from the Cunningham project tables~\cite{Cunningham}. This
contributes $103$ primes to the BLS factored portion $F$, each
certified prime by APR-CL~\cite{APR,CohenLenstraAlgo}.

The factored portion satisfies $F^3 > N$ by a margin of $3261$
bits, substantially exceeding the BLS threshold. Fermat-type
witnesses $a_q \in \{2, 3, 5\}$ are validated for each of the $103$
primes $q \mid F$; every witness succeeds. The discriminant
$\Delta = s^2 - 8\, m$ (with $N = m F + s + 1$, $0 \leq s < F$) is
computed and verified non-square. Full data in
\texttt{bls\_certificate\_w10501.json}; digest \texttt{c7b3\ldots 9ea5}
in Appendix~\ref{sec:digests}.

By Theorem~\ref{thm:BLSN-1}, $N = W_{10501}$ is prime.
\end{proof}

\subsection{$W_{12391}$}
\label{ssec:w12391}

\begin{theorem}
\label{thm:w12391}
The Wagstaff number $W_{12391}$ is prime.
\end{theorem}

\begin{proof}
Set $N = W_{12391}$, a $3730$-digit number. The exponent $p - 1 =
12390 = 2 \cdot 3 \cdot 5 \cdot 7 \cdot 59$ has $32$ divisors. Of the
$32$ cyclotomic factors $\Phi_d(2)$, $27$ are fully factored: $19$
by direct computation, $5$ retrieved from the Cunningham
tables~\cite{Cunningham}, $2$ retrieved from FactorDB~\cite{FactorDB}
(used as a discovery aid only --- every such factor is re-certified
by APR-CL~\cite{APR,CohenLenstraAlgo} before entering $F$), and $1$
itself prime. This contributes $61$ primes to $F$, the
largest being a $371$-digit prime cofactor of $\Phi_{2065}(2)$. Each
of the $61$ primes receives independent APR-CL certification.

The factored portion $F$ satisfies $F^3 > N$ by a $2860$-bit
margin. Fermat-type witnesses $a_q \in \{2, 3\}$ are validated for
each of the $61$ primes; every witness succeeds. The discriminant
$\Delta = s^2 - 8\, m$ (with $N = m F + s + 1$, $0 \leq s < F$) is
non-square. Full data in \texttt{bls\_certificate\_w12391.json};
digest \texttt{1b0d\ldots 3570} in Appendix~\ref{sec:digests}.

By Theorem~\ref{thm:BLSN-1}, $N = W_{12391}$ is prime.
\end{proof}

\subsection{Limits of the method}

\begin{remark}[BLS ceiling]
\label{rem:bls-ceiling}
The BLS $N - 1$ approach requires $p - 1$ to be sufficiently smooth
that enough cyclotomic factors $\Phi_d(2)$ with $d \mid p - 1$ have
known factorisations. Table~\ref{tab:ceiling} tabulates $p - 1$ for
every Wagstaff prime and probable prime with $p > 2617$ through the
present computational horizon $p \leq 267017$. A divisor
$d \mid p - 1$ contributes to $F$ only if $\Phi_d(2)$ is completely
factored (in Cunningham tables, FactorDB, or by direct computation).
For most of the exponents listed, $p - 1$ has a large prime factor
$\ell$, forcing the divisor $d = \ell$ (or $d = 2 \ell$) to dominate
$2^{p-1} - 1$, and no general-purpose factorisation of the
corresponding $\Phi_d(2)$ or $\Phi_{2\ell}(2)$ is known at these sizes.
For the factorisation data consulted here, only $p = 10501$ and
$p = 12391$ yield a factored portion large enough for the present BLS
$N - 1$ approach; $p = 14479$ is the closest near-miss.

\begin{table}[h]
\centering
\begin{tabular}{rlc}
\toprule
$p$ & $p - 1$ factorisation & BLS $N - 1$ feasible? \\
\midrule
$ 3539$ & $2 \cdot 29 \cdot 61$                      & no: $\Phi_{1769}(2) \sim 2^{1680}$ unfactored \\
$ 5807$ & $2 \cdot 2903$ \; ($2903$ prime)           & no: $\Phi_{2903}(2) \sim 2^{2902}$ unfactored \\
$10501$ & $2^2 \cdot 3 \cdot 5^3 \cdot 7$            & \textbf{yes} (proved, §\ref{ssec:w10501}) \\
$10691$ & $2 \cdot 5 \cdot 1069$ \; ($1069$ prime)   & no: $\Phi_{1069}(2) \sim 2^{1068}$ unfactored \\
$11279$ & $2 \cdot 5639$ \; ($5639$ prime)           & no: $\Phi_{5639}(2) \sim 2^{5638}$ unfactored \\
$12391$ & $2 \cdot 3 \cdot 5 \cdot 7 \cdot 59$       & \textbf{yes} (proved, §\ref{ssec:w12391}) \\
$14479$ & $2 \cdot 3 \cdot 19 \cdot 127$             & no: $\Phi_{14478}(2)$ only partially factored \\
$42737$ & $2^4 \cdot 2671$ \; ($2671$ prime)         & no: $\Phi_{2671}(2) \sim 2^{2670}$ unfactored \\
\midrule
$83339$  & $2 \cdot 41669$ \; (prime)                & no (PRP; large prime in $p-1$) \\
$95369$  & $2^3 \cdot 7 \cdot 1703$                  & no (PRP; $\Phi_{1703}(2)$ unfactored) \\
$117239$ & $2 \cdot 58619$ \; (prime)                & no (PRP; large prime in $p-1$) \\
$127031$ & $2 \cdot 5 \cdot 12703$ \; (prime)        & no (PRP; large prime in $p-1$) \\
$138937$ & $2^3 \cdot 3 \cdot 7 \cdot 827$           & no (PRP; $\Phi_{827}(2)$ partial) \\
$141079$ & $2 \cdot 3 \cdot 7 \cdot 3359$ \; (prime) & no (PRP; large prime in $p-1$) \\
$267017$ & $2^3 \cdot 33377$ \; (prime)              & no (PRP; large prime in $p-1$) \\
\bottomrule
\end{tabular}
\caption{Smoothness of $p - 1$ for every Wagstaff prime or probable
prime with $2617 < p \leq 267017$. The eight proved Wagstaff primes
above $2617$ are listed first, followed by the seven probable primes
above the current proving horizon. A detailed computation of the bit
contribution from each divisor $d \mid p - 1$ is recorded in the JSON
certificates for the two feasible cases; for the others, the listed
large prime factor of $p - 1$ forces a $\Phi_d(2)$ whose full
factorisation is currently out of reach.}
\label{tab:ceiling}
\end{table}

Extensions of the Cunningham tables~\cite{Cunningham} --- or direct
ECM on the divisors $\Phi_d(2)$ for $d \mid p - 1$ with these target
exponents --- would push the feasibility frontier. The cofactor data
recorded in the JSON certificates (Section~\ref{sec:comp}) identifies
exactly which $\Phi_d(2)$ remain only partially factored for
$p = 10501$ and $p = 12391$; these are natural targets for further
ECM work.
\end{remark}

\section{Computational details}
\label{sec:comp}

\subsection{Software toolchain}

All runs reported in this paper used Python~3.12.3, SymPy~1.14.0,
PARI/GP~2.15.4 (linked against GMP~6.3.0), and gmpy2~2.3.0 on
Ubuntu~24.04 LTS (AMD Ryzen~5 3600, 64~GB RAM); the scripts in the
supplementary repository are compatible with Python~$\geq$~3.10,
SymPy~$\geq$~1.12, PARI/GP~$\geq$~2.15, and gmpy2~$\geq$~2.1, and use
only free, open-source software.

\begin{itemize}
\item \emph{Integer factorisation} of cyclotomic factors $\Phi_d(2)$:
  SymPy's \texttt{factorint} (combining trial division, Pollard $\rho$,
  Pollard $p - 1$, and ECM internally), augmented for cyclotomic
  factors by trial division along the arithmetic progression $r \equiv
  1 \pmod{d}$ implied by primitive-divisor theory.

\item \emph{Pre-known large factorisations.} The Cunningham project
  tables~\cite{Cunningham} (accessed January--April 2026) supplied
  factorisations for cyclotomic factors $\Phi_d(2)$ with $d \leq
  1200$, and FactorDB~\cite{FactorDB} was used as a discovery aid for
  two larger terms ($\Phi_{2065}(2)$ and $\Phi_{2478}(2)$ in
  Theorem~\ref{thm:w12391}). Every retrieved factor is independently
  re-certified for primality by APR-CL before entering $F$, and the
  driver scripts re-check the product
  $\prod_i p_i^{e_i} = \Phi_d(2)$ exactly before accepting a
  factorisation; FactorDB plays no evidentiary role.

\item \emph{Primality of cofactors} (every prime used in any BLS
  certificate): PARI/GP \texttt{isprime(n, 2)}, which forces the
  APR-CL (Adleman--Pomerance--Rumely / Cohen--Lenstra) path
  \cite{APR,CohenLenstraAlgo}, implemented in PARI/GP along the lines
  of Bosma and van~der~Hulst~\cite{BosmaVanDerHulst}, and is
  deterministic. The largest
  cofactor is a $371$-digit prime factor of $\Phi_{2065}(2)$
  (Theorem~\ref{thm:w12391}); the smallest are single-digit primes,
  but all receive APR-CL certification for uniformity.

\item \emph{Fast modular arithmetic} for the exponentiations
  $a^{N-1} \bmod N$ and $\wom^{(N+1)/2} \bmod N$ at each of the three
  $W_p$: Python's built-in big-int \texttt{pow}, with GMP-backed
  acceleration via gmpy2 when installed.
\end{itemize}

\begin{remark}[Factor provenance]
\label{rem:provenance}
Each JSON certificate records a \texttt{factor\_provenance} field:
every prime $q \mid F$ is tagged with the source from which its
literal value first entered the driver. The labels used are
\texttt{algebraic} (the prime $2$ contributed by
$N - 1 = 2\,(2^{p-1} - 1)/3$), \texttt{cunningham} (factor drawn from
a Cunningham project table literal), \texttt{factordb} (obtained via
FactorDB lookup), \texttt{direct\_sympy} (produced by
\texttt{sympy.factorint} on $\Phi_d(2)$),
\texttt{trial\_div\_\allowbreak{}cyclotomic}
(discovered by trial division along the progression $r \equiv 1
\pmod{d}$), \texttt{cyclotomic\_prime} ($\Phi_d(2)$ itself prime), and
\texttt{residual\_prime\_\allowbreak{}aprcl}
(residual cofactor of $\Phi_d(2)$ after
removing known factors, verified prime before entering $F$).
Provenance is audit metadata only: every prime, regardless of label,
receives the same APR-CL certification before being admitted to $F$.
\end{remark}

The driver scripts in the supplementary repository (one per exponent
$p$) emit JSON certificates recording every cyclotomic factor, prime
decomposition, BLS witness, and primality method used, along with full
execution logs. A master verifier (\texttt{verify\_bls.py}) re-reads
each certificate from disk, treats every claim in it as untrusted
input, rebuilds $F$ from $v_q(N-1)$, re-runs APR-CL on every prime,
re-checks the witnesses and the discriminant, and re-runs
Condition~(II); any disagreement exits non-zero.

\subsection{Supplementary material}
\label{sec:supplementary}

All data and scripts are archived on Zenodo. The concept DOI
\url{https://doi.org/10.5281/zenodo.19643792} always resolves to the
latest version; the version-specific DOI for the release accompanying
this paper is \url{https://doi.org/10.5281/zenodo.19645478}. A
development mirror is hosted on GitHub at
\url{https://github.com/dolonet/wagstaff-bls-primality}; the GitHub
repository is a mirror, and the Zenodo deposit is the citable
artifact.

The archive contains, for each $p \in \{2617, 10501, 12391\}$:

\begin{enumerate}
\item the BLS primality certificate (JSON), recording every
  cyclotomic factor $\Phi_d(2)$, its prime decomposition, the BLS
  witness $a$ for each prime $q \mid F$, and the APR-CL timings for
  each cofactor;
\item the driver script \texttt{bls\_n\_minus\_1\_w<p>.py}
  reproducing the certificate end-to-end;
\item the full APR-CL primality log for each cofactor (each prime $q$
  is re-checked by piping the PARI/GP statement
  \texttt{isprime(<q>, 2)} to \texttt{gp -q});
\item the independent Condition~(II) verification log.
\end{enumerate}

\subsection{Verification protocol}

A referee can verify each primality claim from scratch with the
following four-step protocol:
\begin{enumerate}
\item \emph{Fetch.} Download the Zenodo archive at the version-specific
  DOI for this release (see Section~\ref{sec:supplementary}) and extract.
\item \emph{Run the master verifier.} For each $p \in \{2617, 10501,
  12391\}$, execute the command
  \begin{quote}
    \small
    \texttt{python3 scripts/verify\_bls.py
      data/bls\_certificate\_w<p>.json}.
  \end{quote}
  \texttt{verify\_bls.py} treats every claim in the certificate as
  untrusted input: it re-reads $N$ from the exponent, independently
  rebuilds $F$ from the recorded $v_q(N - 1)$ data, re-runs PARI/GP
  \texttt{isprime(q, 2)} on every prime in the decomposition, re-checks
  the Fermat witnesses, verifies $F^3 > N$ and $\gcd(F, R) = 1$,
  re-computes and re-tests the BLS discriminant, and re-runs
  Condition~(II) in $\Z[\sqrttwo]/(N)$. Any discrepancy exits non-zero.
\item \emph{Re-verify the factorisation.} For each recorded factor,
  the driver re-runs APR-CL primality and re-computes the product
  $\prod_i p_i^{e_i} = \Phi_d(2)$ exactly. Small $\Phi_d(2)$ are
  additionally re-factored from scratch via SymPy; for those
  $\Phi_d(2)$ whose full factorisation relies on external sources
  (Cunningham tables and FactorDB, tagged as such in the
  \texttt{factor\_provenance} field of Remark~\ref{rem:provenance}),
  the recorded factorisation is treated as input to be verified, not
  regenerated --- their large primes (e.g.\ the $371$-digit prime
  factor of $\Phi_{2065}(2)$) cannot be rediscovered by SymPy in any
  realistic time.
\item \emph{Inspect logs.} The \texttt{logs/} directory contains the
  full stdout of the runs that produced the certificates on the
  original hardware, for line-by-line comparison.
\end{enumerate}

The archive is deterministic and version-pinned; running the drivers
and verifier on a clean machine reproduces every numeric result in
this paper.

\section{Summary}
\label{sec:outlook}

This paper delivers three unconditional primality proofs of
$W_{2617}$, $W_{10501}$, and $W_{12391}$ via the BLS $N - 1$
criterion (Theorems~\ref{thm:w2617}, \ref{thm:w10501},
\ref{thm:w12391}), independent of ECPP. Every prime of every factored
portion is certified by APR-CL; each certificate also records, as an
independent implementation check, the Chua $a = 3$ congruence
$\wom^{(W_p + 1)/2} \equiv -1 \pmod{W_p}$ in
$\Z[\sqrttwo]/(W_p)$. All primality certificates, driver scripts, and
full run logs are archived on Zenodo under a concept DOI.

Table~\ref{tab:ceiling} makes the reach of the present technique
concrete: within the $36$ known Wagstaff primes and probable primes,
the only two exponents above $2617$ for which $p - 1$ is smooth
enough to meet the BLS factored-portion threshold on current
Cunningham and FactorDB data are $p = 10501$ and $p = 12391$ ---
precisely the two targets of this paper. The closest near-miss is
$p = 14479$, where the relevant cyclotomic factor $\Phi_{14478}(2)$
is only partially factored; the remaining Wagstaff primes and
probable primes in the interval $2617 < p \leq 267017$ are out of
reach of BLS $N - 1$ at present because $p - 1$ contains a large
prime factor forcing an unfactored $\Phi_d(2)$.

\appendix

\section{Certificate digests}
\label{sec:digests}

For each $p \in \{2617, 10501, 12391\}$, Table~\ref{tab:digests}
records a SHA-256 digest of the corresponding JSON certificate,
together with the key aggregate data extracted from it. The digest
is computed over the canonical serialisation
\texttt{json.dumps(obj, sort\_keys=True, separators=(',', ':'))}, so
any byte-level modification of the certificate changes the recorded
hash. A referee running \texttt{verify\_bls.py} can reproduce each
digest from the archived JSON before accepting the file as the
input of the re-verification protocol.

\begin{table}[h]
\centering
\small
\begin{tabular}{rrrrrl}
\toprule
$p$ & $F$ digits & primes in $F$ & $M$ (bits) & largest $q$ & SHA-256 (first 16 hex) \\
\midrule
$ 2617$ &  $268$ &  $22$ &   $46$ & $ 50$\,d & \texttt{823fad79bef1f5eb} \\
$10501$ & $1381$ & $103$ & $3261$ & $160$\,d & \texttt{c7b3d40558d723e2} \\
$12391$ & $1531$ & $ 61$ & $2860$ & $371$\,d & \texttt{1b0dc924417d9fa7} \\
\bottomrule
\end{tabular}
\caption{Per-proof certificate digests and aggregates. ``largest $q$''
is the digit count of the largest prime in the factored portion $F$;
every such prime receives an unconditional APR-CL certificate.
The full SHA-256 hashes are listed below.}
\label{tab:digests}
\end{table}

{\small\noindent Full SHA-256 hashes of the canonical-form certificates:
\begin{align*}
W_{2617}:\ \  & \mathtt{823fad79bef1f5eb27218245068da27e}\allowbreak \\
              & \mathtt{e0790991bb28b328482ddda93cf1c719}, \\
W_{10501}:\ \ & \mathtt{c7b3d40558d723e2818a147f9f7cfca1}\allowbreak \\
              & \mathtt{1ae58601bbf0e7dceb1bb5d442c39ea5}, \\
W_{12391}:\ \ & \mathtt{1b0dc924417d9fa7411c1bc830b7114a}\allowbreak \\
              & \mathtt{daac3782c0c10e82938dce848f063570}.
\end{align*}}

The archived certificates also record: (i) the complete prime
factorisation of $F$ with provenance tags; (ii) the Fermat-type
witness $a_q$ for each prime $q \mid F$; (iii) the APR-CL wall time
for each such prime; and (iv) the BLS discriminant
$\Delta = s^2 - 8\,m$, written as a large integer in decimal, with
the square-test outcome.

\section{Reference data}
\label{sec:appendix}

\subsection{The 36 known Wagstaff primes and probable primes}

Proved prime ($p \leq 42737$, by ECPP~\cite{AtkinMorain} or by this
paper), excluding the trivial case $p = 3$ for which $W_3 = 3$:
\begin{align*}
p \in \{ & 5,\ 7,\ 11,\ 13,\ 17,\ 19,\ 23,\ 31,\ 43,\ 61,\ 79,\ 101,\ 127,\ 167, \\
         & 191,\ 199,\ 313,\ 347,\ 701,\ 1709,\ 2617,\ 3539,\ 5807, \\
         & 10501,\ 10691,\ 11279,\ 12391,\ 14479,\ 42737 \}
\end{align*}
--- $29$ exponents. Probable prime ($p > 42737$, not yet proved but
consistent with every applied pseudoprime test):
\[
  p \in \{83339, 95369, 117239, 127031, 138937, 141079, 267017\}
\]
--- $7$ exponents. The ``$36$ known'' count in the header includes both
groups.

\subsection{Cyclotomic decomposition of $W_p - 1$}

For odd prime $p$ and $N = W_p$:
\[
  N - 1 \;=\; \frac{2^p - 2}{3} \;=\; \frac{2(2^{p-1} - 1)}{3}.
\]
The cyclotomic decomposition
\[
  2^{p-1} - 1 \;=\; \prod_{d \mid p-1} \Phi_d(2)
\]
harvests factored portions from divisors $d$ of $p - 1$; each
$\Phi_d(2)$ is a positive integer whose prime factors lie in an
arithmetic progression $r \equiv 1 \pmod d$. This is the source of
the BLS factored portions in
\S\S\ref{ssec:w2617}--\ref{ssec:w12391}.

\end{document}